\documentclass[12pt,a4paper]{article}

\usepackage[OT1]{fontenc}
\usepackage[utf8]{inputenc}

\usepackage{hyperref}

\usepackage[english]{babel}
\usepackage{amsfonts}
\usepackage{amsthm}
\usepackage{graphicx}
\usepackage{mathtools}  % (location: texlive-latex3)

\newtheorem{teo}{Theorem}[section]
\newtheorem{prop}[teo]{Proposition}
\newtheorem{lem}[teo]{Lemma}
\newtheorem{coro}[teo]{Corollary}
\newtheorem{defi}[teo]{Definitions}
\newtheorem{rem}[teo]{Remark}

\def\h{{\cal H}}

\def\bh{{\cal B}({\cal H})}
\def\kh{{\cal K}({\cal H})}

\def\ifi{{\cal I}_{\phi}}

\def\ran{{\rm{Ran\, }}}

\def\be{b_{\varepsilon}}

\def\pp1{\overline{p_1}}
\def\ppp1{\overline{\overline{p_1}}}

\begin{document}

\title{\vspace*{0cm} Young's (in)equality for compact operators.}

% 15A42 Linear and multilinear algebra; matrix theory --  Inequalities involving eigenvalues and eigenvectors

% 15A45 Linear and multilinear algebra; matrix theory  -- Miscellaneous inequalities involving matrices

% 47A30  Operator Theory -- Norms (inequalities, more than one norm, etc.)

% 47A63 Operator Theory --  Operator inequalities

\date{}
\author{G. Larotonda\footnote{Supported by Instituto Argentino de Matem\'atica (CONICET), Universidad Nacional de General Sarmiento and ANPCyT.}}

\maketitle

\setlength{\parindent}{0cm} %% para que no indente los parrafos nuevos

\begin{abstract}
If $a,b$ are $n\times n$ matrices, Ando proved that Young's inequality is valid for their singular values: if $p>1$ and $1/p+1/q=1$, then
$$
\lambda_k(|ab^*|)\le \lambda_k\left( \frac1p |a|^p+\frac 1q |b|^q \right) \, \textit{ for all }k.
$$
Later, this result was extended for the singular values of a pair of compact operators acting on a Hilbert space by Erlijman, Farenick and Zeng. In this paper we prove that if $a,b$ are compact operators, then equality holds in Young's inequality if and only if $|a|^p=|b|^q$, obtaining a complete characterization of such $a,b$ in relation to other (operator norm) Young inequalities.\footnotesize{\noindent }

\end{abstract}

\section{Introduction}\label{intro1}

It all boils down to the following elementary inequality named after W. H. Young: if $p>1$ and $1/p+1/q=1$, then for any $\alpha,\beta\in\mathbb R^+$, 
$$\alpha\beta\le \frac{1}{p}\alpha^p+\frac{1}{q}\beta^q
$$
with equality if and only if $\alpha^p=\beta^q$. 

\medskip

Operator analogues of this elegant fact are considered, following the fundamental paper by T. Ando \cite{ando} for $n\times n$ matrices, and an extension for compact operators by J. Erlijman, D. R. Farenick, R. Zeng \cite{efz}. If $a,b$ are compact operators on Hilbert space then for all $k\in\mathbb N_0$,
$$
\lambda_k(|ab^*|)\le \lambda_k\left(\frac1p |a|^p +\frac1q |b|^q\right)
$$
where each eigenvalue is counted with multiplicity. This allows to construct a partial isometry $u$ such that
$$
u |ab^*|u^*\le \frac1p |a|^p +\frac1q |b|^q
$$
for the partial order of operators. Therefore, this raised the natural question of whether
$$
\|u |ab^*|u^*\|=\| \frac1p |a|^p +\frac1q |b|^q\|
$$
implies $|a|^p=|b|^q$. It is easy to construct examples where this is false, if $\|\cdot\|$ is the operator norm. But  O. Hirzallah and F. Kittaneh showed with an elegant inequality \cite{hk} that it is true if $a,b$ are Hilbert-Schmidt operators, and the norm is the Hilbert-Schmidt norm.  Another nice paper, this time  by M. Argerami and D. Farenick \cite{af} proved that that it is also true for $|a|^p,|b|^q$ nuclear operators, that is when the norm is the trace norm $\|\cdot\|_1=Tr|\cdot|$. 

\medskip

In this paper, we prove that the necessary and sufficient condition is in fact the equality of all singular numbers, which enables us to characterize for exactly which norms the assertion above is true (Theorem \ref{elteo}).

\section{Young's inequality for compact operators}

Let $\mathcal H$ be a complex Hilbert space, and let us denote with $\bh$ the bounded linear operators acting in $\h$. For $y\in\bh$, with $|y|=\sqrt{y^*y}$ we denote the positive square root and then $y=\nu|y|$ is the polar decomposition of $y$. With $\nu:\overline{\ran|y|}\to\overline{\ran y}$ we denote its partial isometry, when necessary, the projection $\nu\nu^*$ onto the closure of the range of $y$ will be denoted by $p_y$.

\medskip

 In the following lemma we collect some results that will be used throughout this paper (and will help us fix the notation):

\begin{lem}\label{ellema}
Let $a,b,x\in \bh$, 
\begin{enumerate}
\item If $b=\nu |b|$ then $\nu^*\nu$ is the orthogonal projection onto the closure of the range of $|b|$, $|b^*|=\nu|b|\nu^*$ and $\nu\nu^*$ is the orthogonal projection onto the closure of the range of $|b^*|$.
\item $|ab^*|=\nu | |a| |b| |\nu^*$ and $\nu^*|ab^*|\nu=||a||b||$.
\item If $p$ is a projection, then $x=pxp$ implies $x=px$ and in particular $p_xp=p_x$ (equivalently $p_x\le p$).
\item If $p$ is a projection, $pxp=p$ and either $x\ge p$ or $0\le x\le p$, then $xp=p$. In particular if $\ran(p)=span(\xi)$ for some $\xi \in\mathcal H$, 
$$
\langle x\eta,\eta \rangle\ge \langle p\eta,\eta\rangle \textit{ for any } \eta\in\mathcal H
$$
and $\langle x \xi ,\xi\rangle =\langle \xi,\xi\rangle$ imply $x\xi=\xi$. There is a similar assertion for the other case.
\end{enumerate}
\end{lem}
\begin{proof}
1. is trivial, to prove 2. write the polar decompositions $a=u|a|,b=\nu|b|$. Note that $|ab^*|^2=\nu|b||a|^2|b|\nu^*$; since $\nu^*\nu|b|=|b|$ then  $(\nu|b||a|^2|b|\nu^*)^n=\nu(|b||a|^2|b|)^n\nu^*$ for any $n\in\mathbb N$, and an elementary functional calculus argument shows that
$$
|ab^*|=(\nu|b||a|^2|b|\nu^*)^{\frac12}=\nu(|b||a|^2|b|)^{\frac12}\nu^*=\nu | |a||b||\nu^*.
$$
On the other hand, since $\nu\nu^*\nu=\nu$, then $\nu\nu^*|ab^*|=|ab^*|=|ab^*|\nu\nu^*$, therefore from $\nu^*|ab^*|^2\nu=||a||b||^2$ taking square roots and using a similar argument we derive $\nu^*|ab^*|\nu=||a||b||$

3. If $pxp=x$ then $\ran x\subset \ran p$, therefore $p_x\le p$ or equivalently $pp_x=p_x$. Multiplying both sides by $x$ gives $x=px$.

4. Assume $x\ge p$ (the case $0\le x\le p$ can be treated in a similar fashion therefore its proof is omitted). Since $x-p\ge 0$, we have, for each $\eta\in\mathcal H$,
$$
\|(x-p)^{1/2}p\eta\|^2=\langle p(x-p)p\eta,\eta\rangle=0,
$$
thus $(x-p)^{1/2}p=0$ and multipliying by $(x-p)^{1/2}$ on the left we obtain $(x-p)p=0$ wich shows that $xp=p$.
\end{proof}

\subsection{Singular values}

Denote with  $\kh$  the compact operators on $\mathcal H$. Let $\lambda_k(x)$ ($k\in\mathbb N_0$) denote the $k$-th eigenvalue of the positive compact operator $x\in\bh$, arranged in decreasing order, 
$$
\|x\|=\lambda_0\ge \lambda_1\ge\cdots \lambda_k\ge\lambda_{k+1}\ge \cdots
$$
where we allow equality because each singular value is counted with multiplicity. Clearly $
\lambda_k(f(x))=|f|(\lambda_k(x))$ for any function defined in $\sigma(x)$.

\begin{rem}\label{autov}
For given $a,b\in \bh$ and $x\in \kh$, the min-max characterization of the singular values \cite[Theorem 1.5]{simon} and Lemma \ref{ellema}.2 easily imply that
\begin{enumerate}
\item $\lambda_k(axb)\le \|a\|\|b\|\lambda_k(x)$,
\item $\lambda_k(|ab^*|)=\lambda_k(||a||b||)$.
\end{enumerate}
\end{rem}

%\section{Operator Young inequality}%\label{compa}

\subsection{Unitarily invariant norms}

For a given vector $a=(a_i)_{i\in\mathbb N_0}$ with $a_i\in\mathbb R$, we will denote with $a^{\downarrow}$ the rearrangement of $a$ in decreasing order, that is $a^{\downarrow}$ is a permutation of $a$ such that
$$
a_0^{\downarrow}\ge a_1^{\downarrow}\ge \cdots\ge a_k^{\downarrow}\ge a_{k+1}^{\downarrow}\ge \cdots.
$$

\medskip

Let $\|\cdot\|_{\phi}$ stand for a unitarily invariant norm in $\bh$, and $\phi:{\mathbb R}_+^{\mathbb N_0}\to\mathbb R_+$ its associated permutation invariant gauge \cite{simon}. 

\begin{defi}
We say that the norm is \textbf{strictly increasing} if, given two sequences $a=(a_i),b=(b_i)$ such that $0\le a_i\le b_i$ for all $i\in\mathbb N_0$ and $\phi(a)=\phi(b)$ implies that $a_i=b_i$ for all $i\in\mathbb N$ (see Hiai's paper \cite{hiai}, it is also property (3) in Simon's paper \cite{s2}). Examples of these norms on $\kh$ are the Schatten $p$-norms $1\le p<\infty$, and examples of non-strictly increasing norms are the supremum norm and the Ky-Fan norms. 

Note that we can always define
\begin{equation}\label{normastr}
\phi(a)=\sum\limits_{k\ge 0} a_k^{\downarrow} 2^{-k},
\end{equation}
which is a strictly increasing norm defined in the whole of $\kh$.
\end{defi}

\medskip

\begin{rem}
If $\ifi$ is not equivalent to the supremum norm $\|x\|=\sup\limits_{\|\xi\|_H=1}\|x\xi\|_H$, then
$\ifi=\{x\in\kh: \|x\|_{\phi}<\infty\}$ is a proper bilateral ideal in $\kh$ according to Calkin's theory. Assume that a symmetric norm has the Radon-Riesz property 
$$
\|x_n\|_{\phi}\to \|x\|_{\phi} \,\textit{ and } \, x_n\rightarrow x \,\textit{ weakly }\, \Longrightarrow\, \|x-x_n\|_\phi\to 0
$$
(see Arazy's paper \cite{a} on the equivalence for sequences and compact operators). Simon proved in \cite{s2} that in that case the norm is strictly increasing according to our definition. It is unclear for us if the assertion can be reversed.
\end{rem}

\subsection{Inequality}

\begin{rem}\label{nota} For given $a,b\in\kh$ we will always denote 
$$
\alpha_k=\lambda_k(|a|),\quad \beta_k=\lambda_k(|b|), \quad \gamma_k=\lambda_k(|ab^*|),\quad \delta_k=\lambda_k\left(\frac1p |a|^p +\frac1q |b|^q\right).
$$
Moreover, we will denote
$$
|a|=\sum_k \alpha_k a_k,\quad |b|=\sum_k\beta_k b_k,\quad |ab^*|=\sum_k \gamma_k p_k,\quad \frac1p |a|^p +\frac1q |b|^q=\sum_k \delta_k q_k
$$
the spectral decompositions of each operator, with $a_k,b_k,$etc. one dimensional projections. Note that we allow multiplicity, and if $\gamma_1=\cdots=\gamma_j$ for some finite $j$, the election of the first $q_j$ is arbitrary (i.e., it amounts to select an orthonormal base of that span).
\end{rem}

\begin{rem}\label{aef}
Concerning $a,b\in  \kh$, the following was proved in \cite{efz} by Erlijman, Farenick and Zeng: for each $k\in\mathbb N$,
$$
\lambda_k(|ab^*|)\le \lambda_k\left(\frac1p |a|^p +\frac1q |b|^q\right)
$$
hence there exists a partial isometry $u$ such that $q_k=up_ku^*$ and  $u^*u=\sum_k p_k$ the projection on the (closure of) the range of $|ab^*|$. Then, for any $a,b\in \kh$,
$$
u|ab^*|u^*\le \frac1p |a|^p +\frac1q |b|^q.
$$
This extended the original result of T. Ando \cite{ando} which was stated for positive matrices.
\end{rem}

\medskip

From their result, it can be deduced that the relevant condition to deal with the equality is $\gamma_k=\delta_k$ for all $k$, to be more precise:

\begin{lem}\label{gammaks}Let $a,b\in\kh$,  $p>1$, $1/p+1/q=1$.
\begin{enumerate}
\item If $|a|^p=|b|^q$, then $\alpha_k^p=\beta_k^q=\gamma_k=\delta_k$ for each $k\in\mathbb N_0$.
\item If either
$$
z|ab^*|z^* =\frac1p |a|^p +\frac1q |b|^q
$$ 
for some contraction $z\in \bh$, or 
$$
\|z|ab^*|w\|_\phi=\|\frac1p |a|^p +\frac1q |b|^q\|_\phi,
$$
for a pair of contractions $z,w\in\bh$ and a strictly increasing norm, then $\gamma_k=\delta_k$ for each $k\in\mathbb N_0$ and 
$$
u|ab^*|u^* =\frac1p |a|^p +\frac1q |b|^q.
$$
where $u$ is the partial isometry $u$ of the result in Remark \ref{aef}, i.e $up_ku^*=q_k$ for each $k$.
\end{enumerate}
\end{lem}
\begin{proof}
To prove the first assertion, note that clearly $\delta_k=\alpha_k^p=\beta_k^q$. By Lemma \ref{ellema}.2, $|ab^*|=\nu|a|^p\nu^*$, and in particular 
$$
\lambda_k(|ab^*|)\le\lambda_k(|a|^p)=\lambda_k(|b|^q)
$$
by Remark (\ref{autov}.1). Now since $\nu$ is the partial isometry of $|b|$, then also $\nu^*|ab^*|\nu=\nu^*\nu|a|^p\nu^*\nu=|a|^p$ which in turn shows the reversed inequality, and then $\gamma_k=\alpha_k^p=\beta_k^q$ follows.

Regarding 2., note that if equality is attained by a contraction $z$, then by Remark \ref{aef} and Remark \ref{autov}.1
$$
\gamma_k\le \delta_k=\lambda_k(z|ab^*|z^*)\le \lambda_k(|ab^*|)=\gamma_k.
$$
Likewise, if equality is attained for a strictly increasing norm and a pair of contraction $z,w$, since
$$
\lambda_k(z|ab^*|w)\le \lambda_k(|ab^*|)=\gamma_k\le \delta_k=\lambda_k,
$$
then $\gamma_k=\delta_k$ for every $k$. \end{proof}

\subsection{Equality}

The following result will be crucial to obtain the proof of our main assertion.

\begin{prop}\label{pnoes2}
Let $0\le a,b\in\kh$. Let $1<p<2$ and $1/p+1/q=1$. If
$$
\lambda_k(|ab|)=\lambda_k\left(\frac1p a^p +\frac1q b^q\right)\quad \textit{ for all }k
$$
then $\ran|ba|\subset \overline{\ran b}$.
\end{prop}
\begin{proof}
Let $\varepsilon >0$, let $p_b$ stand for the projection to the closure of the range of $b$, let
$\be=b+\varepsilon(1-p_b)$, then $\be^q=b^q+\varepsilon^q(1-p_b)\le b^q+\varepsilon^q$ and $\be^2=b^2+\varepsilon^2(1-p_b)$. Therefore
$$
|\be a|^2=a\be^2 a= ab^2a +\varepsilon^2 a(1-p_b)a.
$$
Let $|ba|=\sum\limits_{k\in\mathbb N_0}\gamma_k e_k\otimes e_k$ with $\gamma_k=\lambda_k(|ba|)$ and $\{e_k\}_k$ an orthonormal basis of $\ran|ba|$. Then since $\gamma_0=\|ab\|=\|ba\|$, we have $\langle ab^2ae_0,e_0\rangle=\|ba\|^2$ and
\begin{eqnarray}
\varepsilon^2\|(1-p_b)ae_0\|^2+\|ba\|^2 &\le & \|\be a\|^2=\|a\be\|^2\le \|\frac1p a^p+\frac1q \be^q\|^2\nonumber\\
&=  & \|\frac1p a^p+\frac1q b^q+\frac1q \varepsilon^q (1-p_b)\|^2\nonumber\\
&\le &\|\frac1p a^p+\frac1q b^q+\frac1q \varepsilon^q\|^2  =  \left[\|\frac1p a^p+\frac1q b^q\|+\frac1q \varepsilon^q\right]^2\nonumber\\
&  = & (\|ab\|+\frac1q\varepsilon^q)^2= \|ab\|^2+\frac2q\|ab\|\varepsilon^q+\frac{1}{q^2}\varepsilon^{2q}\nonumber
\end{eqnarray}
by Remark \ref{aef} and the hypothesis. Therefore, dividing by $\varepsilon^2$ and letting $\varepsilon\to 0$, since $q>2$, we conclude that $(1-p_b)ae_0=0$ or equivalently, $ae_0\in \overline{\ran b}$. 

We iterate the argument above for all $k$ such that $\gamma_k=\gamma_0$: let us abuse the notation and assume then that $\gamma_1<\gamma_0$. Then for all sufficiently small $\varepsilon\le \varepsilon_1$, 
$$
\gamma_1^2+\varepsilon^2\|(1-p_b)ae_1\|^2<\gamma_0^2.
$$
Therefore for all such $\varepsilon$, if $Q=e_0+e_1$ then
\begin{eqnarray}
\lambda_1\left(Q(\varepsilon^2a(1-p_b)a+ ab^2a)Q\right)&=&\lambda_1\left(\varepsilon^2\|(a(1-p_b)a)^{1/2}e_1\|^2e_1+\gamma_1^2e_1+\gamma_0^ 2e_0\right)\nonumber\\
&=&\gamma_1^2+\varepsilon^2\|(1-p_b)ae_1\|^2.\nonumber
\end{eqnarray}

Now by the same reasons as above (and since $\lambda_k(QAQ)\le \lambda_k(A)$ and $\lambda_k(A+tP)\le \lambda_k(A+t1)=\lambda_k(A)+t$ for $A\ge 0$, $t\in \mathbb R_{\ge 0}$ and $P^2=P=P^*$)
\begin{eqnarray}
\lambda_1\left(Q(\varepsilon^2a(1-p_b)a+ ab^2a)Q\right) &= & \lambda_1|Q|a\be|^2Q|\le \lambda_1|a\be|^2\le \lambda_1\left(\frac1p a^p+\frac1q \be^q\right)^2\nonumber\\
&\le &\lambda_1\left(\frac1p a^p+\frac1qb^q+\frac1q \varepsilon^q\right)^2\nonumber\\
& = & \left[\lambda_1\left(\frac1p a^p+\frac1qb^q\right) +\frac1q \varepsilon^q\right]^2\nonumber\\
&\le &\left(\lambda_1|ab|+\frac1q\varepsilon^q\right)^2=\gamma_1^2+\frac2q\gamma_1\varepsilon^q+\frac{1}{q^2}\varepsilon^{2q}.\nonumber
\end{eqnarray}
Therefore
$$
\gamma_1^2+\varepsilon^2\|(1-p_b)ae_1\|^2\le \gamma_1^2+\frac2q\gamma_1\varepsilon^q+\frac{1}{q^2}\varepsilon^{2q}
$$
and again, dividing by $\varepsilon$ and letting $\varepsilon\to 0$, we conclude that $ae_1\in \overline{\ran b}$. Proceeding recursively, we conclude that $a(\ran|ba|)\subset \overline{\ran b}$. Now if $\xi\in \mathcal H$, then $a|ba|\xi\in\overline{\ran(b)}$, therefore $a^2|ba|\xi=a(a|ba|\xi)\in a\overline{\ran(b)}\subset \overline{\ran(ab)}=\overline{\ran|ba|}$, and $a^3|ba|\xi=a(a^2|ba|\xi)\in a \overline{\ran|ba|}\subset \overline{\ran(b)}$.  Iterating this argument, we arrive to the conclusion that $a^{2n+1}(\ran|ba|)\subset\overline{\ran(b)}$ for all $n\in\mathbb N_0$. Using an approximation of $f=\chi_{\sigma(a)}$ by odd functions, we conclude that $p_a(\ran|ba|)=f(a)(\ran|ba|)\subset\overline{\ran(b)}$ where $p_a$ is the projection onto the closure of the range of $a$. Therefore $|ba|^2\xi=ab^2a\xi=p_aab^2a\xi=p_a|ba|^2\xi\subset \overline{\ran(b)}$,  which gives $\overline{\ran|ba|}=\overline{\ran(|ba|^2)}\subset \overline{\ran(b)}$. 
\end{proof}

\medskip

\begin{rem}\label{proyecciones}
Here are two remarks on projections, its verifications are left to the reader.
\begin{enumerate}
\item Let $b=b^*\in \bh$, $\eta\in \mathcal H$. Then $b(\eta\otimes \eta)b=(b \eta)\otimes(b\eta)$ and the projection onto $span(bx)$ is given by $\frac{(b\eta)\otimes(b\eta)}{\|b\eta\|^2}$.

\item Let $b=b^*\in \bh$ assume that $b\eta=\xi$, with $\|\xi\|=1$. Name $p$ the projection onto $\xi$, name $p_{\eta}$ the projection onto $\eta$. Then 
$$
\|\eta\|^2bp_{\eta}b=p\quad \textit{ and } \quad p_{\eta}b^2p_{\eta}=\frac{1}{\|\eta\|^2}p_{\eta}.
$$
\end{enumerate}
\end{rem}

\begin{lem}\label{minij}
Let $0\le x\in \kh$, let $\xi\in\mathcal H$ with $\|\xi\|=1$. Then
$$
\langle x^r \xi,\xi\rangle  \le \langle x\xi,\xi\rangle^r, \qquad  0<r<1
$$
with equality iff  $x\xi=\langle x\xi,\xi\rangle\xi$. Also
$$
\langle x \xi,\xi\rangle^s  \le \langle x^s\xi,\xi\rangle, \qquad  1<s
$$
with equality iff  $x\xi=\langle x\xi,\xi\rangle\xi$.
\end{lem}
\begin{proof}
Let $x=\sum x_i p_i$ be a spectral decomposition of $x$, with $\sum_i p_i=1$. Let $t_i=\|p_i\xi\|^2$, then $\sum_i t_i =1$. Using H\"older's inequality for sequences, with $p=1/r>1$, we obtain
$$
\langle x^r \xi,\xi\rangle =\sum_i x_i^r t_i = \sum_i x_i^r t_i^{1/p}\, t_i^{1/q}\le \left(\sum_i x_i t_i \right)^r\left(\sum_i t_i\right)^{1/q}=\langle x\xi,\xi\rangle^r.
$$
Assuming equality, in H\"older's inequality, it must be $x_it_i=ct_i$ for all $i$, therefore  $x\xi=\langle x\xi,\xi\rangle \xi$. Taking $s=1/r$ and replacing $x$ with $x^s$, the proof of the other case ($s>1$) is straightforward.
\end{proof}

\medskip

A rewriting of the lemma above, gives the following:

\begin{coro}\label{holderes}
Let $q$ be a rank one projection and $0\le x\in \kh$ then 
$$
qx^rq\le (qxq)^r,\qquad  0<r< 1
$$
with equality iff $xq=cq$ for some $c\ge 0$ and
$$
(qxq)^s\le qx^sq,\qquad 1<s,
$$
with equality iff $xq=cq$.
\end{coro}

\medskip

What follows is the statement that tells us that the relevant hypothesis is neither on the operator equality, nor on the norm equality, but the singular numbers equality.

\begin{teo}\label{igual}
Assume that $a,b\in\kh$, $p>1$, $1/p+1/q=1$. If
$$
\lambda_k(|ab^*|)=\lambda_k\left(\frac1p |a|^p +\frac1q |b|^q\right)
$$ 
for all $k\in\mathbb N_0$, then $|a|^p=|b|^q$.
\end{teo}
\begin{proof}
Since $\lambda_k(|ab^*|)=\lambda_k(|ba^*|)$, exchanging $a$ with $b$ if necessary we can assume that $1<p\le 2$. It will be easier to deal first with $a,b\ge 0$. We follow the notation of Remark \ref{nota}. 

Since $ba^2b=|ab|^2=\sum_k\gamma_k^2 p_k$, if $p_0=\xi\otimes\xi$ with $\xi\in\mathcal H$ and $\|\xi\|=1$, then $ba^2b\xi=\gamma_0^2\xi$. Let $\eta=\frac{1}{\gamma_0^2}p_b a^2b\xi$; then  $\eta\in \overline{\ran b}$ and $b\eta=\xi$. Let $p_{\eta}$ be the projection onto $span(\eta)$, then $\|\eta\|^2bp_{\eta}b=p_0$ by Remark \ref{proyecciones}.2. Observe that 
\begin{equation}\label{aaa}
ba^2b \ge \gamma_0^2 p_0= \gamma_0^2\|\eta\|^2 bp_{\eta} b.
\end{equation}
Now we have to deal with two cases separately, regarding whether $p=2$ or $p\ne 2$.

\textbf{Case $p\ne 2$}. By Proposition \ref{pnoes2}, we have $p_b|ba|=|ba|$, but $\ran|ba|=\ran(ab)$, hence if we name $\overline{a}=p_bap_b$, then
$$
b\overline{a}^2b=bp_bap_bp_bap_bb=bap_bab=ba^2b\ge \gamma_0^2\|\eta\|^2 bp_{\eta} b.
$$
Therefore $\overline{a}^2\ge \gamma_0^2\|\eta\|^2 p_{\eta}$ as operators acting on $\mathcal H'=\overline{\ran b}$. Since $1/2<p/2<1$, the operator monotony of $t\mapsto t^{p/2}$ implies that in $\mathcal H'$, we have
$$
\overline{a}^p\ge \gamma_0^p\|\eta\|^p p_{\eta}.
$$
This also implies
\begin{equation}\label{ap}
\frac{\langle \overline{a}^p\eta ,\eta\rangle}{\|\eta\|^2} \ge \gamma_0^p \|\eta\|^p.
\end{equation}

On the other hand, by Remark \ref{proyecciones}.2 and Corollary \ref{holderes} with $s=q/2> 1$,
\begin{equation}\label{bestr}
\frac{1}{\|\eta\|^q}p_{\eta}=\left(\frac{p_{\eta}}{\|\eta\|^2}\right)^{q/2}=\left( p_{\eta} b^2p_{\eta}\right)^{q/2}\le p_{\eta} b^q p_{\eta},
\end{equation}
equivalently
\begin{equation}\label{bq}
\frac{1}{\|\eta\|^q}\langle \eta, \eta\rangle \le  \langle b^q \eta,\eta \rangle.
\end{equation}
By Young's numeric inequality
\begin{equation}\label{yn}
\gamma_0=\gamma_0\|\eta\|\frac{1}{\|\eta\|}\le \frac1p\gamma_0^p\|\eta\|^p+\frac1q \frac{1}{\|\eta\|^q}.
\end{equation}

Since $1<p<2$, the map $t\mapsto t^p$ is operator convex \cite[Theorem 2.4]{pedersen}, therefore $\overline{a}^p=(p_bap_b)^p\le p_ba^pp_b$, hence combining this with (\ref{ap}), (\ref{bq}) and (\ref{yn}) gives
\begin{eqnarray}\label{t}
\gamma_0  &\le & \frac1p \frac{\langle \overline{a}^p \eta,\eta \rangle}{\|\eta\|^2}   +  \frac1q\frac{\langle b^q \eta,\eta\rangle}{\|\eta\|^2}\le \frac1p \frac{\langle p_b a^p p_b \eta,\eta \rangle}{\|\eta\|^2}   +  \frac1q\frac{\langle b^q \eta,\eta\rangle}{\|\eta\|^2} \nonumber\\
&= & \frac1p \frac{\langle a^p \eta,\eta \rangle}{\|\eta\|^2}   +  \frac1q\frac{\langle b^q \eta,\eta\rangle}{\|\eta\|^2}=\frac{1}{\|\eta\|^2}\left\langle\left( \frac1p a^p +\frac1q b^q \right)\eta,\eta\right\rangle\le \gamma_0\nonumber
\end{eqnarray}
by the hypothesis on the $\lambda_k$.

From here we can derive several conclusions. The first one, since there is equality in Young's numeric inequality (\ref{yn}), is that $\gamma_0=\frac{1}{\|\eta\|^q}$. The second one, since we have equality in (\ref{bestr}), is that $b^2\eta=\frac{1}{\|\eta\|^2}\eta=\gamma_0^{2/q}\eta$ (Lemma \ref{minij}), therefore $\xi=b\eta=\gamma_0^ {1/q}\eta$. The third one, since $0\le  \frac1p a^p +\frac1q b^q\le \gamma_0 1$ and now
$$
\frac{1}{\|\eta\|^2}\left\langle\left( \frac1p a^p +\frac1q b^q \right)\eta,\eta\right\rangle=\frac{1}{\|\xi\|^2}\left\langle\left( \frac1p a^p +\frac1q b^q \right)\xi,\xi\right\rangle=\gamma_0
$$
is that (Lemma \ref{ellema}.4)
$$
\left(\frac1p a^p+\frac1q b^q\right) \xi=\gamma_0 \xi
$$
and rearranging if necessary the basis of  $ker((\frac1p a^p +\frac1q b^q)-\gamma_0 1)$,   we conclude $p_0=p_{\eta}=p_{\xi}=q_0$. Note that
$$
\gamma_0\xi=\frac1p a^p\xi+\frac1q b^q\xi=\frac1p a^p\xi +\frac1q \gamma_0\xi,
$$
which implies that $a^p\xi=\gamma_0\xi$; with a similar argument and since 
$$
0\le \frac1p \overline{a}^p+\frac1q b^q\le \frac1p p_ba^pp_b+\frac1q b^q= p_b( \frac1p p_ba^pp_b+\frac1q b^q)p_b\le \gamma_0 p_b\le \gamma_0 1
$$
we deduce that $\overline{a}^p=\gamma_0\xi$ also, therefore $a\xi=\gamma_0^{1/p}\xi$. 

We now proceed with an induction argument. Write
$$
a=\sum\limits_{\overline{\alpha}_j>\gamma_0^{1/p}}\overline{\alpha}_j \overline{a_j}+ \sum\limits_{\alpha_k\le \gamma_0^{1/p}}\alpha_k a_k, 
$$
with $a_k,\overline{a}_j$ rank one disjoint projections and $a_k\overline{a}_j=0$ for all $k,j$. Then rearranging if necessary  $\alpha_0=\gamma_0^{1/p}$, $a_0=p_0$. Write similarly
$$
b=\sum\limits_{\overline{\beta}_j>\gamma_0^{1/q}}\overline{\beta}_j \overline{b_j}+ \sum\limits_{\beta_k\le \gamma_0^{1/q}}\beta_k b_k, \quad \beta_0=\gamma_0^{1/q},\, b_0=p_0.
$$
Let $\overline{a}=(1-p_0)a(1-p_0)$ and $\overline{b}=(1-p_0)b(1-p_0)$, then  $p_0\overline{a}=p_0\overline{b}=0$,
$$
a=\overline{a}+\gamma_0^{1/p} p_0,\qquad b=\overline{b}+\gamma_0^{1/q} p_0,
$$
$$
ab=\overline{a}\overline{b}+\gamma_0p_0, \quad |ab|=|\overline{a}\overline{b}|+\gamma_0 p_0,
$$
and
$$
\frac1p \overline{a}^p +\frac1q \overline{b}^q+\gamma_0p_0=\frac1p a^p +\frac1q b^q.
$$
Therefore 
$$
\lambda_0\left(\frac1p \overline{a}^p+\frac1q \overline{b}^q\right)=\lambda_1\left(\frac1p a^p+\frac1q b^q\right)=\lambda_1(|ab|)=\lambda_0(|\overline{a}\overline{b}|),
$$
and iterating the above construction we arrive to
$$
a=\overline{a}+\sum_{k}\gamma_k^{1/p} p_k,\quad b=\overline{b}+\sum_k\gamma_k^{1/q} p_k 
$$
with $\overline{a}p_k=\overline{b}p_k=0$ for each $j,k\in\mathbb N_0$. Then
$$
\frac1p a^p+\frac1q b^q = \sum_k \lambda_k p_k + \frac1p \overline{a}^p+\frac1q \overline{b}^q=|ab|+ \frac1p \overline{a}^p+\frac1q \overline{b}^q=|ab|+T
$$
with $T\ge 0$ compact and $T|ab|=0$.  Now $\lambda_k(|ab|)=\lambda_k(\frac1p a^p+\frac1q b^q)$ for all $k$, which means equal eigenvalues with equal (and finite) multiplicities, a fact that forces $T=0$, therefore $\overline{a}=\overline{b}=0$, from which the claim $a^p=b^q$ follows for $a,b\ge 0$, assuming $1<p<2$. 

\medskip

\textbf{Case $p=2$}. Let us now return to the case we skipped. From (\ref{aaa}), we know that $p_ba^2p_b\ge \gamma_0^2\|\eta\|^2p_\eta$ on the whole $\mathcal H$, therefore
\begin{eqnarray}
\gamma_0  &\le & \frac12\frac{\gamma_0^2\|\eta\|^2}{\|\eta\|^2}+\frac12\frac{1}{\|\eta\|^2}\le \frac12 \frac{\langle p_ba^2p_b \eta,\eta \rangle}{\|\eta\|^2}   +  \frac12\frac{\langle b^2 \eta,\eta\rangle}{\|\eta\|^2}=\frac{1}{\|\eta\|^2}\left\langle\left( \frac12 p_ba^2p_b +\frac12 b^2 \right)\eta,\eta\right\rangle\nonumber\\
&= &  \frac{1}{\|\eta\|^2}\left\langle\left(p_b\left( \frac12 a^2 +\frac12 b^2 \right)p_b\right)\eta,\eta\right\rangle =\frac{1}{\|\eta\|^2}\left\langle\left( \frac12 a^2 +\frac12 b^2 \right)\eta,\eta\right\rangle \le \gamma_0\nonumber
\end{eqnarray}
since $\eta\in\ran(b)$. Then from the equality in the numerical inequality (\ref{yn}) we derive that $\lambda_0=\|\eta\|^{-2}$, and $(\frac12 a^2 +\frac12 b^2)\eta=\gamma_0\eta$ as before. Since $q=2$, we have lost the strict inequality in (\ref{bestr}) regarding $b$. However, Since now $\langle p_ba^2p_b \eta,\eta  \rangle\ge \gamma_0^2{\|\eta\|^2}$ must be an equality, from Lemma \ref{ellema}.4 we conclude that $p_ba^2\eta=p_ba^2p_b\eta=\lambda \eta$ for some positive $\lambda$, hence $b^2\eta=(2\gamma_0-\lambda)\eta$ also. Recalling $1=\|\xi\|^2=\|b\eta\|^2=\langle b^2\eta,\eta\rangle=(2\gamma_0- \lambda)\|\eta\|^2=(2\gamma_0-\lambda)\gamma_0^{-1}$, we obtain $\lambda=\gamma_0$. This tells us that $b\eta=\gamma_0^{1/2}\eta=a\eta$. The rest of the argument follows as in the case of $p<2$. 

Returning to the original statement, if for arbitrary compact $a,b$, we have equality of singular values, since $\lambda_k(|ab^*|)=\lambda_k(||a||b||)$ (Remark \ref{autov}.2), we obtain $|a|^p=|b|^q$.
\end{proof}

\bigskip

Let us resume all the results in one clear cut statement, the main result of this paper:

\begin{teo}\label{elteo}
Let $a,b\in \kh$. If $p>1$ and $1/p+1/q=1$, then the following are equivalent:
\begin{enumerate}
\item $|a|^p=|b|^q$.
\item $z|ab^*|z^* =\frac1p |a|^p +\frac1q |b|^q$ for some contraction $z\in \bh$ 
\item $\|z|ab^*|w\|_{\phi} =\|\frac1p |a|^p +\frac1q |b|^q\|_{\phi}$ for a pair of contractions $z,w\in\bh $ and $\|\cdot\|_{\phi}$ a \textbf{strictly increasing} symmetric norm. 
\item $\lambda_k(|ab^*|)=\lambda_k\left(\frac1p |a|^p+\frac1q |b|^q\right)$ for all $k\in\mathbb N_0$.
\end{enumerate}
\end{teo}
\begin{proof}
Clearly $1\Rightarrow 2$ with $z=\nu$ (the partial isometry in the polar decomposition of $b=\nu|b|$). If $2$ holds, picking a norm as in equation (\ref{normastr}), we have $2\Rightarrow 3$. By Lemma \ref{gammaks}, we have  $3\Rightarrow 4$ and finally, by Theorem \ref{igual} it follows that  $4\Rightarrow 1$.
\end{proof}

\subsubsection{Final remarks: equality of operators}

Assume that we have an equality of operators
\begin{equation}\label{igutrcomp}
z|ab^*|z^*=\frac{1}{p}|a|^p+\frac{1}{q}|b|^q
\end{equation}
for some contraction $z\in\bh$. Then from the previous theorem $|a|^p=|b|^q$ and
$$
z|b^*|^qz^*=z\nu|b|^q\nu^*z^*=z|ab^*|^*z^*=|b|^q.
$$ 
\begin{rem}
Let $Tr$ stand for the semi-finite trace of $\bh$. Assume for a moment that $Tr|b|^q<\infty$, or equivalently, that $\beta_k=\lambda_k(b)\in \ell_q$. Then
$$
Tr(|b^*|^q(1-z^*z))=Tr|b^*|^q-Tr(z|b^*|^q)=Tr|b|^q-Tr(z|b^*|^qz^*)=0,
$$
which is only possible if $|b^*|^q=|b^*|^qz^*z$, since $z$ is a contraction and the trace is faithful. Then also
$$
zz^*|b|^q=zz^*z|b^*|^qz^*=z|b^*|^qz^*=|b|^q,
$$
and 
$$
|b|^qz=z|b^*|^qz^*z=z|b^*|^q
$$
or equivalently $|b|z=z|b^*|$, which can be stated as $bz\nu=\nu zb$. The reader can check that these  three conditions 
$$
1) \; |b^*|z^*z=|b^*|,\qquad 2)\; |b|zz^*=|b|,\qquad 3)\; |b|z=z|b^*|
$$
are also sufficient to have equality in (\ref{igutrcomp}).
\end{rem}

This last fact, for $z$ a partial isometry (and with a different proof) was observed  \cite{af} by Argerami and Farenick.

\medskip

\textit{We conjecture that these three conditions are also necessary for (\ref{igutrcomp}) to happen with a contraction $z$ if $b$ is just compact}.

\subsection*{Acknowledgements}

I would like to thank Jorge Antezana for pointing me to the nice paper by Argerami and Farenick; I would also like to thank Esteban Andruchow for our valuable conversations on this subject, and Martin Argerami for all his help improving the manuscript.

\bigskip

\noindent
Gabriel Larotonda\\
Instituto de Ciencias \\
Universidad Nacional de General Sarmiento \\
J. M. Gutierrez 1150 \\
(B1613GSX) Los Polvorines \\
Buenos Aires, Argentina  \\
e-mail: glaroton@ungs.edu.ar

\end{document}